\documentclass[12pt,a4paper,leqno]{article}

\usepackage{latexsym}
\usepackage{amssymb,exscale}
\usepackage[centertags]{amsmath}
\usepackage{amsthm}
\usepackage{graphicx}

\usepackage[dvips]{hyperref}

\newtheorem{Theorem}{Theorem}
\swapnumbers
\newtheorem{theorem}[equation]{Theorem}
\newtheorem{lemma}[equation]{Lemma}
\theoremstyle{definition}

\renewcommand{\phi}{\varphi}

\newcommand{\D}{\mathrm{d}}
\newcommand{\E}{\mathrm{e}}

\newcommand{\ti}{\tilde}

\renewcommand{\(}{\bigl(}
\renewcommand{\)}{\bigr)\vphantom{)}}

\newcommand{\const}{\operatorname{const}}

\newcommand{\eps}{\varepsilon}

\newcommand{\si}{\sigma}
\newcommand{\ga}{\gamma}

\newcommand{\al}{\alpha}
\newcommand{\be}{\beta}

\newcommand{\Ex}{\mathbb E\,}

\newcommand{\Z}{\mathbb Z}

\renewcommand{\Pr}[1]{\mathbb{P}\mskip1.5mu\(\mskip1.5mu#1\mskip1.5mu\)}
\newcommand{\PR}[1]{\mathbb{P}\mskip1.5mu\bigg(\mskip1.5mu#1\mskip1.5mu\bigg)}

\begin{document}

\title{Moderate deviations on different scales:\\
 no relations}

\author{Boris Tsirelson}

\date{}
\maketitle

\begin{abstract}
An example of a discrete-time stationary random process whose sums
follow the normal approximation within a given part of the region of
moderate deviations, but violate it outside this part.
\end{abstract}

\setcounter{tocdepth}{1}
\tableofcontents

\section*{Introduction}
\addcontentsline{toc}{section}{Introduction}
In the simplest classical case of sums $ S_n = X_1+\dots+X_n $ of
bounded i.i.d.\ random variables $ X_k $ with mean $ 0 $ and variance
$ 1 $, CLT (Central limit theorem) states the normal approximation
\begin{equation}\label{*}
\PR{ \frac{S_n}{\sqrt n} > u } \to \int_u^\infty \frac1{\sqrt{2\pi}}
\E^{-t^2/2} \, \D t \quad \text{as } n \to \infty \, ;
\end{equation}
MDP (moderate deviations principle) extends the normal approximation
to larger $ u $,
\begin{equation}\label{**}
\ln \PR{ \frac{S_n}{\sqrt n} > u_n } \sim -\frac12 u_n^2 \quad
\text{as } n \to \infty
\end{equation}
whenever $ u_n \to \infty $ and $ n^{-1/2} u_n \to 0 $ (``$\sim$''
means that the ratio converges to $1$); and LDP (large deviations
principle) treats even larger $ u $,
\begin{equation}\label{***}
\ln \PR{ \frac{S_n}{\sqrt n} > a\sqrt n } \sim -n I(a) \quad
\text{as } n \to \infty \, ,
\end{equation}
where the rate function $ I(\cdot) $ satisfies
\begin{equation}\label{****}
I(a) \sim \frac12 a^2 \quad \text{as } a \to 0 \, .
\end{equation}

It may be tempting to deduce \eqref{**} from \eqref{***}
and \eqref{****} by taking $ a_n \sqrt n = u_n $; however, \eqref{***}
does not claim uniformity in $ a $ (for small $ a $). Under some
additional assumptions, MDP can indeed be deduced from LDP,
see \cite[Th.~1.2]{Wu}.

We consider a discrete-time stationary random process $ (X_k)_k
$. That is, ($ X_k $ need not be independent, and) for every $ n $ the
joint distribution of $ X_m, \dots, X_{m+n} $ does not depend on $ m
$.

Numerous works prove LDP and/or MDP under various conditions on $
(X_k)_k $. Still, in general, relations between LDP and MDP, as well
as between MDP for different sequences $ (u_n)_n $, are poorly
understood.

\begin{Theorem}\label{Th1}
Assume that an open set $ G_0 \subset (0,1/2) $, bounded away from $ 0
$, consist of a finite number of intervals, and $ G_1 $ is the
interior of the complement $ (0,1/2) \setminus G_0 $. Then there
exists a stationary process $ (X_k)_k $ such that for all $ \ga \in
G_0 \cup G_1 $ and $ c \in (0,\infty) $,
\[
\lim_{n\to\infty} \frac1{n^{2\ga}} \ln \PR{ \frac{X_1+\dots+X_n}{\sqrt
n} > c n^\ga } = \begin{cases}
 -\frac12 c^2 &\text{if } \ga \in G_1, \\
 0 &\text{if } \ga \in G_0.
\end{cases}
\]
\end{Theorem}

In addition, the process $ (X_k)_k $ constructed in the proof of
Theorem \ref{Th1} satisfies \eqref{***}, \eqref{****}. Also, the
distribution of $ X_1 $ (and every $ X_k $) is bounded and symmetric,
that is, $ \Pr{ |X_1| \le C } = 1 $ for some $ C $, and $ \Pr{ X_1 \le
x } = \Pr{ -X_1 \le x } $ for all $ x $. Also, $ \Ex (X_1 X_n) =
O \( \exp(-n^\eps) \) $ (as $ n \to \infty $) for some $ \eps > 0 $.

\numberwithin{equation}{section}

\section[Constructing a process]
  {\raggedright Constructing a process}
\label{sect1}
We start with a very simple Markov chain well-known in the renewal
theory. Its state space is the countable set $ \{(0,0)\} \cup
\( \{1,2,3,\dots\} \times \{1,2,3,\dots\} \) $.
\[
\begin{gathered}\includegraphics{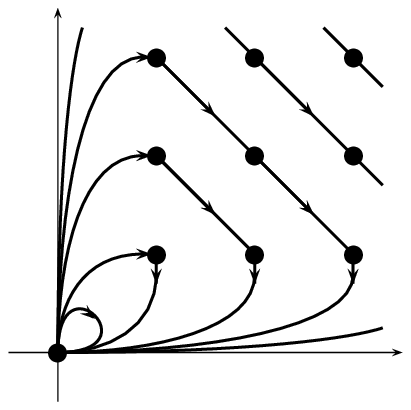}\end{gathered}
\]
The transition probability from $(0,0)$ to $(1,n-1)$ is equal to $ p_n
$ for $ n=2,3,\dots $; and $ p_1 = 1-p_2-p_3-\dots $ is the transition
probability from $(0,0)$ to itself. For $ k>0, l>0 $ the only possible
transition from $(k,l)$ is to $(k+1,l-1)$ if $ l>1 $, otherwise to
$(0,0)$. Thus, the motion consists of excursions of the form $ (0,0)
\mapsto (1,n-1) \mapsto (2,n-2) \mapsto \dots \mapsto (n-1,1) \mapsto
(0,0) $ (for $ n>1 $) or just $ (0,0) \mapsto (0,0) $ (for $n=1$).

An invariant probability measure (if exists) must give equal
probabilities $ \mu_n $ to all points $(k,l)$ with $k+l=n$; here $
n \in \{0\} \cup \{2,3,4,\dots\} $. Such $ \mu $ is invariant if and
only if
\begin{equation}\label{1.*}
\mu_0 p_n = \mu_n \quad \text{for } n=2,3,4,\dots
\end{equation}
(the equality $ \mu_0 p_1 + \mu_2 + \mu_3 + \dots = \mu_0 $
follows). And, of course, $ \mu_0 + \mu_2 + 2\mu_3 + 3\mu_4 + \dots =
1 $, that is, $ \mu_0 (1+p_2+2p_3+\dots) = 1 $ or, equivalently,
\[
\mu_0 ( p_1 + 2p_2 + 3p_3 + \dots ) = 1 \, .
\]
Clearly, $ \mu $ exists if and only if $ p_1 + 2p_2 + 3p_3 + \dots
< \infty $ (finite mean renewal time); and $ \mu $ is unique.

We restrict ourselves to the steady-state Markov chain; that is, the
invariant probability measure $ \mu $ exists, and the state is
distributed $ \mu $ at every instant. Denoting by $ (A_t,B_t) $ the
state of the Markov chain at $ t $ we get a two-dimensional random
process $ (A_t,B_t)_{t\in\Z} $, both stationary and Markovian. The
random set $ \{ t : (A_t,B_t) = (0,0) \} $ is well-known as a
stationary renewal process; its points are called renewal
times. Accordingly, $ A_t $ is called the age (or backwards recurrence
time) at $ t $, and $ B_t $ --- the residual life time (or forward
recurrence time) at $ t $. Note that $ t - A_t $ is the last renewal
time on $ (-\infty,t] $, and $ t + B_t $ --- the first renewal time on
$ [t,\infty) $. The duration $ (t+B_t) - (t-A_t) = A_t + B_t $ of an
excursion (unless $ t $ is a renewal time) is called a renewal
interval.
\[
\begin{gathered}\includegraphics{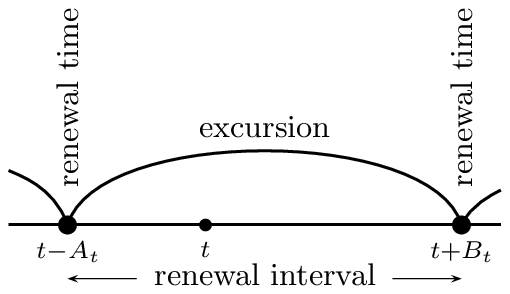}\end{gathered}
\]

Large and moderate deviations will be examined for the stationary
process
\begin{equation}\label{1.1}
X_t = \pm\phi(A_t,B_t) \quad \text{for } t \in \Z
\end{equation}
where a random sign ``$\pm$'' is chosen once for each excursion, and
$ \phi $ is a real-valued function on the state space of the Markov
chain. More formally, we may take $ X_t = S_{t-A_t} \phi(A_t,B_t) $
where $ (S_t)_{t\in\Z} $ is a family of mutually independent
random signs ($-1$ or $+1$ equiprobably), independent of $(A_t,B_t)_t
$. It remains to specify the numbers $ p_n $, or equivalently $ \mu_n
$, and the function $ \phi $.

Given two parameters $ \al, \be $ satisfying
\begin{equation}\label{1.2}
\al > 0 \, , \quad \be \ge 0 \, , \quad \al + 2\be < \frac12 \, ,
\end{equation}
we define the measure $ \mu $ by
\begin{equation}\label{1.3}
n \mu_{n+1} + (n+1) \mu_{n+2} + \dots = \exp \(
-n^\al \) \quad \text{for } n=1,2,\dots
\end{equation}
and the function $ \phi $ by
\begin{equation}\label{1.4}
\phi (k,l) = \begin{cases}
 (k+l)^{-\be} &\text{if } k^2 \le k+l,\\
 0 &\text{otherwise}
\end{cases}
\end{equation}
for $ k,l>0 $; and $ \phi(0,0) = 0 $.

\begin{theorem}\label{1.5}
There exists $ \si \in (0,\infty) $ such that the process $ (\si^{-1}
X_t)_{t\in\Z} $ satisfies Theorem \ref{Th1} for
\begin{gather}
G_0 = (u,v) \, , \quad G_1 = (0,u) \cup (v,0.5) \,
 , \quad \text{where} \notag \\
u = \frac{ \al }{ 2(1-\al-2\be) } \, , \quad v = \frac12 - 2\be \,
 . \label{1.7}
\end{gather}
\end{theorem}

Remark: by \eqref{1.2}, $ 0 < u < \al < v \le 0.5 $.

Remark: arbitrary $ u,v $ satisfying $ 0 < u < v \le 0.5 $ are of the
form \eqref{1.7} for some $ \al, \be $ satisfying \eqref{1.2}. Proof:
take $ \be = 0.25(1-2v) $ and $ \al = \frac{1+2v}{1+2u}u $, then $ 0.5
- \al - 2\be = \frac{v-u}{1+2u} > 0 $.

\section[Examining the process]
  {\raggedright Examining the process}
\label{sect2}
The sum
\[
S_n = X_1 + \dots + X_n
\]
may be treated as the sum over excursions (from the origin to the
origin) of the process $ (A_t,B_t)_t $; no ambiguity appears, since $
\phi(0,0)=0 $.

We introduce $ \ti S_n = X_{1+B_1} + \dots + X_{n-A_n} $ (the sum over
complete excursions within $ [1,n] $), $ S'_n = X_1+\dots+X_{1+B_1} $
and $ S''_n = X_{n-A_n}+\dots+X_n $ (the contributions of the two
incomplete excursions).
\[
\begin{gathered}\includegraphics{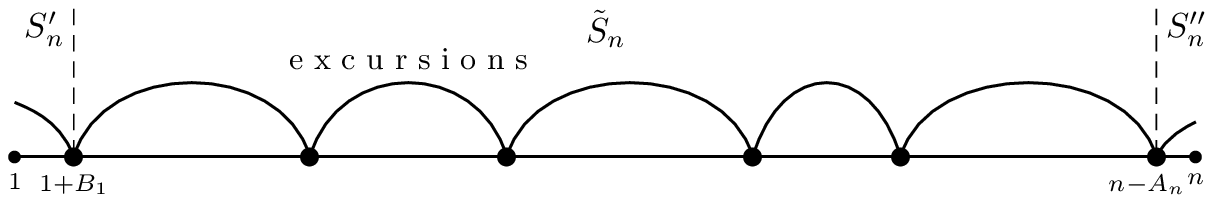}\end{gathered}
\]
It may happen that $ (A_t,B_t) = (0,0) $ only
once on $ [1,n] $; then $ \ti S_n = 0 $. It may also happen that $
(A_t,B_t) \ne (0,0) $ for all $ t \in [1,n] $; in this case we take $
S'_n = 0 $, $ \ti S_n = 0 $, $ S''_n = S_n $. In every case,
\begin{equation}\label{2.1}
S_n = S'_n + \ti S_n + S''_n \, .
\end{equation}
Conditionally, given $ (A_k,B_k)_k $, the summands $ S'_n, \ti S_n,
S''_n $ are independent and symmetrically distributed (due to the
``$\pm$'' in \eqref{1.1}), thus, for every $ c $,
\begin{equation}\label{2.2}
\Pr{ S_n > c } \ge \frac12 \max \Big( \Pr{ S'_n+S''_n>c }, \Pr{ \ti
S_n)>c } \Big) \, .
\end{equation}
It appears that $ \ti S_n $ satisfy MDP for all $ \ga $
(Lemma \ref{2.3}), while $ S'_n + S''_n $ is too large for $ \ga \in
(u,v) $ but small for $ \ga \in (0,u) \cup (v,0.5) $
(Lemma \ref{2.4}).

\begin{lemma}\label{2.3}
There exists $ \si \in (0,\infty) $ such that
\[
\ln \PR{ \frac{\ti S_n}{\si \sqrt n} > u_n } \sim -\frac12 u_n^2
\quad \text{as } n \to \infty
\]
whenever $ u_n \to \infty $ and $ n^{-1/2} u_n \to 0 $.
\end{lemma}

\begin{lemma}\label{2.4}
For all $ \ga \in (0,u) \cup (u,v) \cup (v,0.5) $ and all $ c > 0 $,
\[
\lim_{n\to\infty} \frac1{n^{2\ga}} \ln \PR{ \frac{S'_n+S''_n}{\sqrt
n} > c n^\ga } = \begin{cases}
 -\infty &\text{if } \ga \in (0,u) \cup (v,0.5), \\
 0 &\text{if } \ga \in (u,v).
\end{cases}
\]
\end{lemma}

We first deduce Theorem \ref{1.5} from these lemmas proved afterwards.

\begin{proof}[Proof of Theorem \ref{1.5}]
For $ \ga \in (u,v) $ we have by \eqref{2.2} and Lemma \ref{2.4}
\[
\liminf_{n\to\infty} \frac1{n^{2\ga}}
\ln \PR{ \frac{S_n}{\sqrt n} > c n^\ga } \ge
\liminf_{n\to\infty} \frac1{n^{2\ga}}
\ln \PR{ \frac{S'_n+S''_n}{\sqrt n} > c n^\ga } = 0 \, .
\]
Let $ \ga \in (0,u) \cup (v,0.5) $ and $ c>0 $. On one hand,
by \eqref{2.2} (again) and Lemma \ref{2.3},
\[
\liminf_{n\to\infty} \frac1{n^{2\ga}}
\ln \PR{ \frac{S_n}{\si\sqrt n} > c n^\ga } \ge
\liminf_{n\to\infty} \frac1{n^{2\ga}}
\ln \PR{ \frac{\ti S_n}{\si\sqrt n} > c n^\ga } = -\frac12 c^2 \, .
\]
On the other hand, for every $ \eps>0 $,
\[
\PR{ \frac{S_n}{\si\sqrt n} > c n^\ga } \le
\PR{ \frac{\ti S_n}{\si\sqrt n} > (c-\eps) n^\ga } +
\PR{ \frac{S'_n+S''_n}{\si\sqrt n} > \eps n^\ga } \, ;
\]
for $ n $ large enough we have
\[
\PR{ \frac{\ti S_n}{\si\sqrt n} > (c-\eps) n^\ga } \le \exp \Big(
n^{2\ga} \( -0.5 (c-\eps)^2 + \eps \) \Big)
\]
by Lemma \ref{2.3}; a similar (and even much stronger) bound for
$ \Pr{ \frac{S'_n+S''_n}{\si\sqrt n} > \eps n^\ga } $ holds by
Lemma \ref{2.4}; thus,
\[
\limsup_{n\to\infty} \frac1{n^{2\ga}}
\ln \PR{ \frac{S_n}{\si\sqrt n} > c n^\ga } \le
-\frac12 (c-\eps)^2 + \eps
\]
for all $ \eps>0 $.
\end{proof}

The process $ (\ti S_n)_{n=1,2,\dots} $ is a so-called renewal-reward
process; it jumps at renewal times, and each jump size depends only on
the corresponding (just finished) renewal interval. This process is
delayed (that is, not necessarily starts at a renewal time); a
time-shifted process $ (S_{1+B_1+n})_{n=0,1,\dots} $ is an ordinary
(that is, not delayed) renewal-reward process. MDP for such processes
are available \cite{I} under the conditions
\begin{gather}
\Ex \tau < \infty \, , \label{2.5} \\
\Ex \exp ( \eps X^2 - \tau ) < \infty \quad \text{for some } \eps >
0 \, ; \label{2.6}
\end{gather}
here $ \tau $ is a renewal interval, and $ X $ is the corresponding
jump size. (The renewal-reward process is formed by a sequence of
independent copies of the pair $ (\tau,X) $.)

Here is a result for ordinary renewal-reward processes $
\( S(t) \)_{t\in[0,\infty)} $, formulated in \cite{I} only for $ \Ex
X^2 = 1 $, $ \Ex \tau = 1 $; the general case follows easily by
rescaling.

\begin{theorem}\label{2.7}
\textup{(\cite{I})}
If \eqref{2.5}, \eqref{2.6} are satisfied and $ \Ex X = 0 $, $ \Ex X^2
\ne 0 $ then $ \Ex X^2 < \infty $ and
\[
\lim_{x\to\infty,x/\sqrt t\to0} \frac1{x^2} \ln \Pr{ \si^{-1} S(t) >
x\sqrt t } = -\frac12
\]
where
\[
\si = \sqrt{ \frac{ \Ex X^2 }{ \Ex \tau } } \, .
\]
\end{theorem}

Lattice and nonlattice cases are both covered by Theorem \ref{2.7},
but we need only the lattice case: $ \tau \in \{1,2,\dots\} $ a.s.;
then
\begin{equation}\label{2.8}
\ln \Pr{ \si^{-1} S(n) > u_n \sqrt n } \sim -\frac12 u_n^2
\end{equation}
whenever $ u_n \to \infty $ and $ n^{-1/2} u_n \to 0 $. The transition
from ordinary to delayed process is made as follows.

\begin{lemma}\label{2.9}
Let random variables $ S(n) $ satisfy \eqref{2.8}, and for all $ n $,
$ S(n) \le n $ a.s. Let $ T $ be a random variable independent of $
(S(n))_n $, $ T \in \{0,1,2,\dots\} $ a.s. Then \eqref{2.8} holds also
for the random variables
\[
S'(n) = \begin{cases}
 S(n-T) &\text{for } n \ge T,\\
 0 &\text{otherwise.}
\end{cases}
\]
\end{lemma}

\begin{proof}
For every $ \eps>0 $, for all $ n \ge n_\eps $,
\[
\exp \Big( -\frac{1+\eps}2 u_n^2 \Big) \le \Pr{ \si^{-1} S(n) >
u_n \sqrt n } \le \exp \Big( -\frac{1-\eps}2 u_n^2 \Big) \, .
\]
We choose $ t $ such that $ \Pr{ T=t } > 0 $ and get
\begin{multline*}
\Pr{ \si^{-1} S'(n) > u_n \sqrt n } \ge
 \Pr{ T=t } \Pr{ \si^{-1} S(n-t) > u_n \sqrt n } \ge \\
\ge \Pr{ T=t } \exp \Big( -\frac{1+\eps}2 \frac{n}{n-t} u_n^2 \Big)
\end{multline*}
for all $ n \ge n_\eps + t $, thus,
\[
\limsup_{n\to\infty} \Big( \frac{-2}{u_n^2} \ln \Pr{ \si^{-1} S'(n) >
u_n \sqrt n } \Big) \le 1 + \eps
\]
for all $ \eps>0 $. On the other hand,
\[
\Pr{ \si^{-1} S'(n) > u_n \sqrt n }
= \sum_{k=0}^n \Pr{T=k} \Pr{ \si^{-1} S(n-k) > u_n \sqrt n } \, ;
\]
for $ n $ large enough we have $ \si u_n \sqrt n \ge n_\eps $, thus, $
k < n-n_\eps $ (otherwise $ \si^{-1} S(n-k) $ cannot exceed $
u_n \sqrt n $, since $ S(n-k) \le n-k $ a.s.), therefore
$ \Pr{ \si^{-1} S'(n) > u_n \sqrt n } \le \max_k \Pr{ \si^{-1} S(n-k)
> u_n \sqrt n } \le \exp \( -\frac{1-\eps}2 u_n^2 \) $ and
\[
\liminf_{n\to\infty} \Big( \frac{-2}{u_n^2} \ln \Pr{ \si^{-1} S'(n) >
u_n \sqrt n } \Big) \ge 1 - \eps
\]
for all $ \eps>0 $.
\end{proof}

In our situation
\begin{equation}\label{2.10}
\Pr{ \tau = k } = p_k \quad \text{for } k=1,2,\dots
\end{equation}
(for $ p_k $ see \eqref{1.*}) and roughly
\begin{equation}\label{2.11}
X \approx \pm \sqrt{\tau} \cdot \tau^{-\be}
\end{equation}
(recall \eqref{1.1} and \eqref{1.4}); in order to make it exact,
$ \sqrt\tau $ should be replaced with its integral part, but this
small correction is left to the reader.

\begin{proof}[Proof of Lemma \ref{2.3}]
By \eqref{1.3}, $ \ln \mu_n \sim -n^\al $; by \eqref{1.*}, $ \ln
p_n \sim -n^\al $; by \eqref{2.10}, \eqref{2.11} and \eqref{1.2} we
see that $ \Ex \tau < \infty $, $ \Ex X = 0 $, $ \Ex X^2 \ne 0 $,
and \eqref{2.6} is satisfied (with $ \eps=1 $). Thus,
Theorem \ref{2.7} gives MDP for $ (\ti S_{1+B_1+n})_{n=0,1,\dots} $,
and Lemma \ref{2.9} gives \eqref{2.8} for $ (\ti S_n)_{n=1,2,\dots}
$.
\end{proof}

\begin{proof}[Proof of Lemma \ref{2.4}]
By \eqref{1.3},
\begin{equation}\label{2.12}
\ln \Pr{ A_1+B_1>k } = \ln \Pr{ A_n+B_n>k } = -k^\al \quad \text{for }
k=1,2,\dots
\end{equation}
Similarly to \eqref{2.11},
\begin{equation}\label{2.13}
\pm (A_1+B_1)^\be S'_n \approx
\begin{cases}
 (\sqrt{A_1+B_1}-A_1)^+ &\text{if } 1+B_1 \le n,\\
 0 &\text{otherwise};
\end{cases}
\end{equation}
\begin{multline}\label{2.14}
\pm (A_n+B_n)^\be S''_n \approx \\
 \approx \Big( \min\(n,\,n-A_n+\sqrt{A_n+B_n}\,\) - \max(1,n-A_n+1) + 1 \Big)^+ = \\
= \Big( \min \( A_n, \sqrt{A_n+B_n} \,\) - (A_n-n)^+ \Big)^+ \, ;
\end{multline}
here $ x^+ = \max(x,0) $.
Note that
\begin{equation}\label{2.15}
|S'_n| \le (A_1+B_1)^{0.5-\be} \, , \quad |S''_n| \le
(A_n+B_n)^{0.5-\be} \, .
\end{equation}

Let $ \ga \in (0,u) \cup (u,v) \cup (v,0.5) $ and $ c>0 $ be
given. (Recall \eqref{1.7}.)

Case 1: $ \ga \in (0,u) $.

We have $ \ga < \frac{\al}{ 2(1-\al-2\be) } $, that is,
$ \frac{\ga+0.5}{0.5-\be}\al > 2\ga $. Taking into account that
\begin{multline*}
\PR{ \frac{ S'_n+S''_n }{ \sqrt n } > c n^\ga } \le \Pr{ S'_n > 0.5 c
 n^{\ga+0.5} } + \Pr{ S''_n > 0.5 c n^{\ga+0.5} } \le \\
\le \Pr{ (A_1+B_1)^{0.5-\be} > 0.5 c n^{\ga+0.5} } + \Pr{
(A_n+B_n)^{0.5-\be} > 0.5 c n^{\ga+0.5} }
\end{multline*}
we get
\[
\frac1{n^{2\ga}} \ln \PR{ \frac{ S'_n+S''_n }{ \sqrt n } > c n^\ga
} \le \frac{\ln 2}{n^{2\ga}}
- \frac1{n^{2\ga}} \cdot \const \cdot \Big(
n^{ \frac{\ga+0.5}{0.5-\be} } \Big)^\al \to -\infty \, .
\]

Case 2: $ \ga \in (u,\al] $.

We have $ \ga > \frac{\al}{ 2(1-\al-2\be) } $, that is,
$ \frac{\ga+0.5}{0.5-\be} < \frac{2\ga}{\al} $. For large $ n $, we
choose $ c_n \in \{1,2,\dots\} $ such that
\[
n^{\frac{\ga+0.5}{0.5-\be}} \ll c_n \ll n^{\frac{2\ga}{\al}} \, ,
\]
note that $ \sqrt{c_n} \ll n^{\ga/\al} \le n $, and choose $
a_n,b_n \in \{1,2,\dots\} $ such that $ a_n+b_n = c_n $ and
\[
\sqrt{ a_n+b_n } < a_n < n \, .
\]
Taking into account that $ \ln \mu_n \sim -n^\al $ we have
\[
\ln \Pr{ A_n=a_n, B_n=b_n } = \ln \mu_{a_n+b_n} \sim -(a_n+b_n)^\al \,
.
\]
Conditionally, given $ A_n = a_n $ and $ B_n = b_n $ we have
by \eqref{2.14}
\[
S''_n \approx \pm (a_n+b_n)^{0.5-\be} \, ; \quad |S''_n| \approx
c_n^{0.5-\be} \gg n^{\ga+0.5} \, .
\]
Also, $ S'_n $ and $ S''_n $ are conditionally independent, since $
[1,n] $ contains (at least one) renewal time $ n - a_n $. Thus, for
large $ n $,
\[
\PR{ \frac{ S'_n+S''_n }{ \sqrt n } > c n^\ga } \ge \frac12 \Pr{ S''_n
> c n^{\ga+0.5} } \ge \frac14 \Pr{ A_n=a_n, B_n=b_n } \, ,
\]
and we get
\[
\frac1{n^{2\ga}} \ln \PR{ \frac{ S'_n+S''_n }{ \sqrt n } > c n^\ga
} \ge -\frac{\ln 4}{n^{2\ga}} - \frac{c_n^\al}{n^{2\ga}} \( 1 +
o(1) \) \to 0 \, .
\]

Case 3: $ \ga \in [\al,v) $.

The proof of Case 2 needs only the following modifications. In the
start: we have $ \ga < 0.5-2\be $, that is,
$ \frac{\ga+0.5}{0.5-\be} < 2 $, and we take
\[
n^{\frac{\ga+0.5}{0.5-\be}} \ll c_n \ll n^2 \, .
\]
In the end: $ c_n^\al \ll n^{2\al} \le n^{2\ga} $.

Case 4: $ \ga \in (v,0.5) $.

Combining \eqref{2.15} with trivial inequalities $ |S'_n| \le n
(A_1+B_1)^{-\be} $, $ |S''_n| \le n (A_n+B_n)^{-\be} $ we get
\[
|S'_n| \le (A_1+B_1)^{-\be} \min ( n, \sqrt{A_1+B_1} ) \, , \quad
|S''_n| \le (A_n+B_n)^{-\be} \min ( n, \sqrt{A_n+B_n} ) \, .
\]
We note that generally
\[
x^{-\be} \min(n,\sqrt x) \le n^{1-2\be} \quad \text{for all } x \in
(0,\infty)
\]
(indeed, the function $ x \mapsto x^{-\be} \min(n,\sqrt x) $ is
increasing on $ (0,n^2] $ since $ \be < 0.5 $, and decreasing on $
[n^2,\infty) $ since $ \be \ge 0 $).

Thus, $ |S'_n| \le n^{1-2\be} $, $ |S''_n| \le n^{1-2\be} $ a.s., and
we get for large $ n $
\[
\PR{ \frac{ S'_n+S''_n }{ \sqrt n } > c n^\ga } = 0 \, ,
\]
since $ 0.5 - 2\be = v < \ga $.
\end{proof}

\section[Combining such processes]
  {\raggedright Combining such processes}
\label{sect3}
Given $ 0 < u_1 < v_1 < u_2 < v_2 < 0.5 $, we construct two
independent processes $ (X_t^{(1)})_t $, $ (X_t^{(2)})_t $ as in
Sect.~\ref{sect1}, the former satisfying Theorem \ref{Th1} for $ G_0 =
(u_1,v_1) $, the latter --- for $ G_0 = (u_2,v_2) $. Then their sum
\[
X_t = X_t^{(1)} + X_t^{(2)} \, ,
\]
being normalized (divided by $ \si $), satisfies Theorem \ref{Th1} for
$ G_0 = (u_1,v_1) \cup (u_2,v_2) $.

\begin{proof}
Case 1: $ \ga \in (0,u_1) \cup (v_1,u_2) \cup (v_2,0.5) $.

By Theorem \ref{1.5}, both processes
\[
\si_i^{-1} S_n^{(i)} = \frac{ X_1^{(i)} + \dots + X_n^{(i)}
}{ \si_i \sqrt n }
\]
(for $ i=1,2 $) satisfy LDP with speed $ (n^{2\ga})_n $ and rate
function $ c \mapsto 0.5 c^2 $. By independence, the two-dimensional
process $ \( \si_1^{-1} S_n^{(1)}, \si_2^{-1} S_n^{(2)} \)_n $
satisfies LDP with speed $ (n^{2\ga})_n $ and rate function $
(c_1,c_2) \mapsto 0.5 (c_1^2+c_2^2) $. By the contraction principle,
the linear combination $ (\si^{-1} S_n)_n $ (where $ S_n = S_n^{(1)}
+ S_n^{(2)} $ and $ \si = \sqrt{\si_1^2+\si_2^2} \, $) satisfies LDP with
speed $ (n^{2\ga})_n $ and rate function $ c \mapsto 0.5 c^2 $.

Case 2: $ \ga \in (u_1,v_1) $.

The argument is the same, but now the first rate function is
identically zero, thus, the two-dimensional rate function is $
(c_1,c_2) \mapsto 0.5 c_2^2 $, and the contraction principle returns
identically zero.

Case 3: $ \ga \in (u_2,v_2) $.

Similar to Case 2.
\end{proof}

A finite number of independent processes can be combined similarly,
which completes the proof of Theorem \ref{Th1}.

\bigskip
\filbreak
{
\small
\begin{sc}
\parindent=0pt\baselineskip=12pt
\parbox{4in}{
Boris Tsirelson\\
School of Mathematics\\
Tel Aviv University\\
Tel Aviv 69978, Israel
\smallskip
\par\quad\href{mailto:tsirel@post.tau.ac.il}{\tt
 mailto:tsirel@post.tau.ac.il}
\par\quad\href{http://www.tau.ac.il/~tsirel/}{\tt
 http://www.tau.ac.il/\textasciitilde tsirel/}
}

\end{sc}
}
\filbreak

\end{document}